\documentclass[reqno]{amsart}
\usepackage{amssymb,latexsym}
\usepackage{amsfonts,mathrsfs}
\usepackage[margin=1.4in]{geometry}
\usepackage[all]{xy}

\newtheorem{theorem}{Theorem}[section]
\newtheorem{lemma}[theorem]{Lemma}

\newtheorem{corollary}[theorem]{Corollary}
\newtheorem{definition}[theorem]{Definition}

\theoremstyle{definition}

\newcommand{\ba}{\begin{array}}
\newcommand{\ea}{\end{array}}

\def \qed{\cqfd}

\def \O{\Omega}

\def \g{\gamma}

\def \bi{\bar{i}}
\def \bj{\bar{j}}
\def \bk{\bar{k}}
\def \bl{\bar{l}}

\def \C{\mathbb C}

\DeclareMathOperator \diam {diam}

\newcommand{\Vol}{\mathrm{Vol}}

\def\Xint#1{\mathchoice
{\XXint\displaystyle\textstyle{#1}}%
{\XXint\textstyle\scriptstyle{#1}}%
{\XXint\scriptstyle\scriptscriptstyle{#1}}%
{\XXint\scriptscriptstyle\scriptscriptstyle{#1}}%
\!\int}
\def\XXint#1#2#3{{\setbox0=\hbox{$#1{#2#3}{\int}$ }
\vcenter{\hbox{$#2#3$ }}\kern-.58\wd0}}
\def\intav{\Xint-}

\def\qed{\vbox{\hrule
\hbox{\vrule\hbox to 5pt{\vbox to 8pt{\vfil}\hfil}\vrule}\hrule}}

\newcommand{\beg}{\begin{eqnarray*}}
\newcommand{\begn}{\begin{eqnarray}}
\newcommand{\en}{\end{eqnarray*}}
\newcommand{\enn}{\end{eqnarray}}

\newcommand{\tr}{\mbox{\rm tr\,}}

\begin{document}
\title{A mean value formula and a Liouville theorem for the  complex Monge-Amp\`ere equation}
\author{Chao Li}
\address{School of Mathematical Sciences\\
University of Science and Technology of China\\
Hefei, 230026,P.R. China\\ } \email{leecryst@mail.ustc.edu.cn}
\author{Jiayu Li}
\address{School of Mathematical Sciences\\
University of Science and Technology of China\\
Hefei, 230026\\ and AMSS, CAS, Beijing, 100080, P.R. China\\} \email{jiayuli@ustc.edu.cn}
\author{Xi Zhang}
\address{School of Mathematical Sciences\\
University of Science and Technology of China\\
Hefei, 230026,P.R. China\\ } \email{mathzx@ustc.edu.cn}
\thanks{The authors were supported in part by NSF in
China,  No.11625106, 11571332, 11526212.}

\begin{abstract} In this paper,  we prove a mean value formula for bounded subharmonic Hermitian matrix valued function on a complete Riemannian manifold with nonnegative Ricci curvature. As its application, we obtain a Liouville type theorem  for the complex Monge-Amp\`ere equation
 on product manifolds.
\end{abstract}

\maketitle

\section{Introduction}

\hspace{0.3cm}

Understanding
various spaces of harmonic
functions on complete noncompact Riemannian manifolds is one of the central questions in
geometric analysis. During the last 40 years, there have been many significant progress in this question ( see e.g. \cite{Y1,Y2,CY1,LS,Li86,LT,Li95,CCM95,Wjp}, $\cdots$).
More importantly, the techniques developed in this field are
extremely useful when applied to other problems in geometric analysis.
In \cite{Li95}, Peter Li proved the following theorem:

\medskip

\begin{theorem}[Theorem 2 of \cite{Li95}]\label{li95}Let $(M^n, \omega )$ be a complete K\"ahler manifold with nonnegative
Ricci curvature and $\mathcal{H}^{1}(M)$ be the space of linear growth harmonic functions on $(M^n, \omega )$. Then $dim \mathcal{H}^{1}(M) \leq 2n+1$. Moreover, if $dim \mathcal{H}^{1}(M) = 2n+1$  then $M$ must be isometric to $\mathbb C^n$ with the standard
flat metric.
\end{theorem}

\medskip

In Peter Li's proof of Theorem \ref{li95},  the following mean value theorem for bounded subharmonic functions plays an important role:

\medskip

\begin{theorem}[Lemma B of \cite{Li95}]\label{mvf}Let $(M, g)$ be a complete manifold with nonnegative Ricci curvature.
Suppose f is a bounded subharmonic function defined on $(M, g)$, then for any $p\in M$
\begin{equation}\lim_{r\rightarrow \infty}\intav_{B_r(p)}f dV_{g}=\sup_M f.\end{equation}
\end{theorem}

Despite the application in \cite{Li95}, Theorem \ref{mvf} has some more applications in the study of Riemannian geometry (see e.g. \cite{CCM95}). It is a useful tool in the study of linearly growth harmonic functions on complete Riemannian manifolds with nonnegative Ricci curvature.

In this paper, we study a class of Hermitian matrix valued functions and establish a mean value theorem to them. For convenience, we denote the set of all $m$-order Hermitian matrices by $\mathbf{Hm}(m)$, and equip it with the metric induced by the inner product
\begin{equation}\langle A,B\rangle=\tr A\overline{B}.\end{equation}

\begin{definition} A map $A=(A_{ij})$ from a Riemannian manifold to $\mathbf{Hm}(m)$ is said to be subharmonic, if for any vector $\xi=(\xi_1,\cdots,\xi_m)\in \mathbb{C}^m$, $\xi A\xi^*=A_{ij}\xi_i\overline{\xi_j}$ is a subharmonic function.
\end{definition}

By the definition, it is easy to check that a $\mathcal{C}^2$ Hermitian  matrix valued function $A=(A_{ij})$ on a Riemannian manifold is subharmonic if and only if $\Delta A=(\Delta A_{ij})$ is semi-positive-defined everywhere. We obtain the following mean value formula to subharmonic Hermitian matrix valued functions.

\medskip

\begin{theorem}\label{mvhm} Let $(M, g)$ be a complete Riemannian manifold with nonnegative Ricci curvature, and $A=(a_{ij})$
be a bounded subharmonic Hermitian  matrix valued function on $(M, g)$. Then there exits a Hermitian  matrix $A_0$, such that
\begin{equation}A\leq A_0,\end{equation}
on M, and for any $p\in M$
\begin{equation}\label{mvf1}\lim_{r\rightarrow \infty}\intav_{B_r(p)} A dV_{g}=\lim_{r\rightarrow \infty}\left(\intav_{B_r(p)}a_{ij} dV_{g}\right)= A_0.\end{equation}
\end{theorem}

\medskip

The complex Monge-Amp\`ere equation has  significant applications in complex analysis and complex geometry, and many remarkable progresses of
 complex Monge-Amp\`ere
equations  were carried out by many people (see e.g. \cite{Au,Y78,BT,BT1,CY,MY,Kob,CKNS,TY1,TY2,TY3,TY4,Ti0,Ti1,Ti2,Ko,Gb,Bl1,GPF,ZZ,TWY,DP,DZhZh,ZhZh,WY}, $\cdots$). In this paper, we concentrate on Liouville theorems for the complex Monge-Amp\`ere equation.
In \cite{RS84}, Riebesehl and Schulz proved a Liouville theorem for the complex Monge-Amp\`ere equation on $\mathbb C^{n}$, which can be expressed by K\"ahler forms as following.

\medskip

\begin{theorem}[\cite{RS84}]\label{ltcma} Let $\omega$ be a K\"ahler form on $\mathbb{C}^n$ satisfy
$C^{-1}\omega_0\leq \omega \leq C\omega_0$
and
$\omega^n=\omega_0^n$,
where $\omega_0=\frac{\sqrt{-1}}{2}\sum\limits_{i=0}^n dz^i\wedge d\bar{z}^i$ and $C$ is a positive constant.
Then $\nabla_{\omega_0}\omega=0$, or equivalently
\begin{equation}\omega=\frac{\sqrt{-1}}{2}\sum\limits_{i,j=1}^n A_{ij}dz^i\wedge d\bar{z}^j \end{equation}
for some constant Hermitian matrix $(A_{ij})$.
\end{theorem}

\medskip

The key of the proof of Theorem \ref{ltcma} is a local Calabi $\mathcal{C}^3$ estimate, i.e. an estimate on $|\nabla_{\omega_0}\omega|_{w}^2$.
However, when considering analogous Liouville type theorems on complete K\"ahler manifolds with non trivial Riemanian curvature, the Calabi $\mathcal{C}^3$ estimate seems not to work.
Recently, Hein (\cite{H17}) proved a Liouville theorem for the complex Mong-Amp\`ere eqution on product manifolds, which can be expressed in short as:

\medskip

\begin{theorem}[Theorem A of \cite{H17}]\label{hein17}Let $(Y,\omega_{Y_0})$ be a compact Ricci-flat K\"ahler manifold. Let $\omega$ be a Ricci-flat
K\"ahler form on $\mathbb C^m\times Y$. Assume that $C^{-1}(\omega_{\mathbb C^m} + \omega_{Y_0})\leq \omega \leq C(\omega_{\mathbb C^m} + \omega_{Y_0})$ for some $C>1$, where $\omega_{\mathbb C^m}$ is the standard flat K\"ahler form on $\mathbb C^m$. Then we can find some K\"ahler form $\omega_{Y}$ on $Y$, $T_{l}\in Auto(\mathbb C^m\times Y)$ and complex linear map $S\in Auto(\mathbb C^m)$ such that:
\begin{equation*}T_l^*\omega=\omega_Y+S^*\omega_{C^m}.\end{equation*}
\end{theorem}

\medskip

In Hein's proof of Theorem \ref{hein17}, one key step is to study the converging property of a sequence of subharmonic functions $u_t$ with respect to K\"ahler metrics $\omega_t$ which are constructed from $\omega$. In this paper, we consider  the case that $Ric (\omega_{Y_{0}})\geq 0$ and establish the following Liouville theorem:

\medskip

\begin{theorem}\label{mth}Let $(Y^n,\omega_{Y_0})$ be an $n$ dimensional compact K\"ahler manifold with nonpositive Ricci curvature, and let $\omega$ be a K\"ahler form $\omega$ on $\mathbb{C}^m\times Y$ with properties
\begin{itemize}
\item[1).]$C^{-1}(\omega_{\mathbb{C}^m}+\omega_{Y_0})\leq \omega \leq C(\omega_{\mathbb{C}^m}+\omega_{Y_0})$, for some positive constant $C$;
\item[2).]$\omega^{n+m}=(\omega_{\mathbb{C}^m}+\omega_{Y_0})^{m+n}$.
\end{itemize}
where $\omega_{\mathbb{C}^m}$ is the standard K\"ahler form on $\mathbb{C}^m$.
Then there exists a K\"ahler form $\omega_Y$ on $Y$ with $Ric(\omega_Y)=Ric(\omega_{Y_0})$ such that $\nabla_{\omega_{\mathbb{C}^m}+\omega_Y}\omega=0$.
Furthermore, we have the
the following presentation of $\omega$
\begin{equation}\label{exs}\omega=\hat{\omega}_{\mathbb{C}^n}+\omega_Y+\frac{1}{2}\sum\limits_{i=1}^m(dz^i\wedge \eta^i+d\bar{z}^i\wedge\overline{\eta^i})\end{equation}
where $\hat{\omega}_{\mathbb{C}^n}=\frac{1}{2}\sum\limits_{i,j=1}^m u_{i\bj}dz^i\wedge d\bar{z}^j$ with the constant Hermitian matrix $(u_{i\bj})$, and every $\eta^i$ is a $\omega_Y$-paralleled (0,1)-form.
\end{theorem}

\medskip

Taking the construction of $\omega_Y$, $T_l$ and $S$ in Theorem \ref{hein17} (\cite{H17}) in consideration, Theorem \ref{mth} can be seen as a generalization of Theorem \ref{hein17}. Our proof relies on  the above mean value formula (i.e. Theorem \ref{mvhm}) and is very  different with Hein's. Theorem \ref{ltcma} also can be seen as an application of the mean value formula (\ref{mvf1}). We hope the mean value formula (\ref{mvf1}) to have  more applications in the study of K\"ahler geometry.

\hspace{0.4cm}

\section{A mean value formula for bounded subharmonic Hermitian matrix valued function}

\hspace{0.3cm}

In this section, we first give a proof of Theorem \ref{mvhm} and then give a new proof to  Theorem \ref{ltcma} by using Theorem \ref{mvhm} instead of the Calabi $\mathcal{C}^3$ estimate.

\medskip

{\bf A proof of Theorem \ref{mvhm}. }
For any vector $\xi\in \mathbb{C}^m$, define
\begin{equation}||\xi||_A^2=\xi A\xi^*.\end{equation}
By this definition and the condition on $A$, for any fixed $\xi\in \mathbb{C}^m$, $||\xi||_A^2$ is a bounded subharmonic function, then Theorem \ref{mvf} implies
\begin{equation}\label{mvlim1}\lim_{r\rightarrow \infty}\intav_{B_r(p)}||\xi||_A^2=\sup_M ||\xi||_A^2.\end{equation}
For $i=1,2,\cdots,n$, let $e_i$ be the i-th direction vector in $\mathbb{C}^n$. We have
\begin{equation}A_{ij}=\frac{||e_i+e_j||_A^2-||e_i-e_j||_A^2}{4}-\sqrt{-1}\frac{||e_i+\sqrt{-1}e_j||_A^2-||e_i-\sqrt{-1}e_j||_A^2}{4}.\end{equation}
Together with (\ref{mvlim1}), we assert that $\lim\limits_{r\rightarrow \infty}\intav_{B_r(p)} A$ exists. Let
\begin{equation}A_0=\lim_{r\rightarrow \infty}\intav_{B_r(p)} A,\end{equation}
then we have for any $\xi\in \mathbb{C}^m$,
\begin{equation}\xi A\xi^*=||\xi||_A^2\leq \lim_{r\rightarrow \infty}\intav_{B_r(p)}||\xi||_A^2=\xi A_0\xi^*.\end{equation}
This shows $A\leq A_0$.

\hfill $\Box$ \\

\medskip

We have the following simple corollary:

\medskip

\begin{corollary}\label{mvhm1} Let $A:M\rightarrow\mathbf{Hm}(m)$ satisfy the same condition of Theorem \ref{mvhm} and $A_0=\lim\limits_{r\rightarrow \infty}\intav_{B_r(p)} A$. Let $F$ be a bounded function on some neighborhood of the closure of $A(M)$ and continuous at $A_0$, then we have
\begin{equation}\lim_{r\rightarrow \infty}\intav_{B_r(p)} F(A)=F(A_0).\end{equation}
\end{corollary}

\begin{proof}By the condition on $F$ we can find a positive constant $C$, such that
\begin{equation}F(A)\leq C,\end{equation}
on $M$. And for any $\varepsilon>0$, we can find some $\delta>0$ such that for any $q\in M$ satisfying $|A(q)-A_0|\leq \delta$, there holds
\begin{equation}|F(A(q))-F(A_0)|\leq \varepsilon.\end{equation}
For the mentioned $\varepsilon$ and $\delta$, we have
\begin{equation}\label{cor1}
\begin{aligned}\int_{B_r(p)}|F(A)-F(A_0)|&=\int_{B_r(p)\cap \{|A-A_0|\leq \delta\}}|F(A)-F(A_0)|+\int_{B_r(p)\cap \{|A-A_0|> \delta\}}|F(A)-F(A_0)|\\
&\leq \varepsilon \Vol(B_r(p)\cap \{|A-A_0|\leq \delta\})+C\Vol(B_r(p)\cap \{|A-A_0|> \delta\})\\
&\leq \varepsilon \Vol(B_r(p))+C\Vol(B_r(p)\cap \{|A-A_0|> \delta\}).
\end{aligned}\end{equation}
By Theorem \ref{mvhm}, $A\leq A_0$, so $A_0=\lim\limits_{r\rightarrow \infty}\intav_{B_r(p)} A$ implies
\begin{equation}\lim\limits_{r\rightarrow \infty}\intav_{B_r(p)} |A-A_{0}|=0.\end{equation}
Together with
\begin{equation}\Vol(B_r(p)\cap \{|A-A_0|> \delta\})\leq \delta^{-1}\int_{B_r(p)}|A-A_0|,\end{equation}
we have
\begin{equation}\label{cor2}\lim\limits_{r\rightarrow \infty}\frac{\Vol(B_r(p)\cap \{|A-A_0|> \delta\})}{\Vol(B_r(p))}=0.\end{equation}
(\ref{cor1}) and (\ref{cor2}) imply
\begin{equation}
\limsup\limits_{r\rightarrow \infty}\intav_{B_r(p)}|F(A)-F(A_0)|\leq \varepsilon.
\end{equation}
Let $\varepsilon\rightarrow 0$, then we get
\begin{equation}
\lim\limits_{r\rightarrow \infty}\intav_{B_r(p)}|F(A)-F(A_0)|=0.
\end{equation}
This concludes the proof.
\end{proof}

\hspace{0.4cm}

By Therorem \ref{mvhm} and Corollary \ref{mvhm1} we can give a new proof to Theorem \ref{ltcma}.

\medskip

{\bf A new proof to Theorem \ref{ltcma}. } We can write $\omega$ as
\begin{equation}\omega=\frac{\sqrt{-1}}{2}\sum\limits_{i,j=1}^n u_{i\bj}dz^i\wedge d\bar{z}^j,\end{equation}
where $(u_{i\bj})$ is a function valued in $\mathbf{Hm}(n)$.
Denote $(u^{i\bj})=(u_{i\bj})^{-1}$, $u_{i\bj k}=\frac{\partial}{\partial z^k}u_{i\bj}$, $u_{i\bj k\bl }=\frac{\partial}{\partial z^l}u_{i\bj k}$, etc.
Since $\omega$ is closed, we have
\begin{equation}u_{i\bj k}=u_{k\bj i},\qquad u_{i\bj\bk }=u_{i\bk\bj}.\end{equation}
By the equation $\omega$ satisfied, we have
\begin{equation}\label{cma1}\det(u_{i\bj})=1.\end{equation}
Direct computation shows
\begin{equation}\Delta_{\omega}u_{i\bj}=u^{k\bl}u^{p\bar{q}}u_{k\bar{q}i}u_{\bl p \bj}.\end{equation}
For any $\xi=(\xi^1,\xi^2,\cdots,\xi^n)\in \mathbb{C}^n$, consider the Hermitian quadratic form $F:\mathbb{C}^{3n}\times \mathbb{C}^{3n} \rightarrow \mathbb{R}$ defined by
\begin{equation}(A,B)\mapsto u^{i\bar{\alpha}}u^{\beta\bj}\xi^k\overline{\xi^{\gamma}}A_{ijk}\overline{B_{\alpha\beta\gamma}}.\end{equation}
By choosing a proper frame on $\mathbb{C}^n$, one can eaily check that $F$ is semi-positive-defined. So
\begin{equation}\xi^i(\Delta_{\omega}u_{i\bj})\overline{\xi^j}=u^{i\bar{\alpha}}u^{\beta\bj}\xi^k\overline{\xi^{\gamma}}u_{i\bj k}\overline{u_{\alpha \bar{\beta}\gamma}}\geq 0.\end{equation}
This implies that $(u_{i\bj})$ is subharmonic. The condition on $\omega$ implies that $(u_{i\bj})$ is bounded and $(\mathbb{C}^n,\omega)$ is a complete Ricci flat K\"ahler manifold . By Theorem \ref{mvhm} and Corollary \ref{mvhm1}, we can find a constant Hermitian matrix $A$ such that
\begin{equation}\label{cma2}(u_{i\bj})\leq A,\end{equation}
on $M$ and
\begin{equation}\lim_{r\rightarrow \infty}\intav_{B^{\omega}_r(O)}\det(u_{i\bj})\omega^n=\det A.\end{equation}
By (\ref{cma1}) and the previous equality, we have
\begin{equation}\det A =\det(u_{i\bj})= 1.\end{equation}
Since $(u_{i\bj})$ is positive-defined, the previous equality and (\ref{cma2}) imply $(u_{i\bj})$ is the constant function $A$.
This concludes the proof.

\hfill $\Box$ \\

\medskip

{\bf Remark: } {\it 1). To prove $(u_{i\bj})$ is subharmonic, besides direct computation, we can also use the following argument:
For any $\xi\in \mathbb{C}^m$, let $X_{\xi}=\xi_i\frac{\partial}{\partial z^i}$, then
\begin{equation}\xi_i u_{i\bj} \overline{\xi_j}=2|X_{\xi}|_{\omega}^2.\end{equation}
Using the Bochner formula for holomorhic fields and the fact that $Ric(\omega)= 0$, one can easily check that $\xi_i u_{i\bj}\overline{\xi_j}$ is subharmonic.

2). To prove Theorem \ref{ltcma}, one can also consider the $(u^{i\bj})=(u_{i\bj})^{-1}$. For any $\xi\in \mathbb{C}^m$, let $f_{\xi}=Re(\xi_iz^i)$, then
\begin{equation}\xi_iu^{i\bj}\overline{\xi_j}=|df_{\xi}|^2_{\omega}.\end{equation}
Clearly $f_{\xi}$ is a pluri-harmonic function and hence a harmonic function with respect to $\omega$. Using the Bochner formula and the fact that $Ric(\omega)= 0$ one can easily check that $\xi_iu^{i\bj}\overline{\xi_j}$ is subharmonic.}

\hspace{0.2cm}

\section{A Liouville theorem for the  complex Monge-Amp\`ere equation}

\hspace{0.3cm}

In this section, we obtain a Liouville theorem for the  complex Monge-Amp\`ere equation as an application of the mean value formula (\ref{mvf1}), i.e. we give a proof of Theorem \ref{mth}.
First we introduce the following lemma concerning the computation of determine of a blocked Hermitian matrix.

\medskip

\begin{lemma}\label{bmdetinv}Let $M$ be an invertible Hermitian matrix. If \begin{equation*}M=\left(\begin{array}{cc}A&C\\C^*&B\end{array}\right),\qquad M^{-1}=\left(\begin{array}{cc}\widetilde{A}&\widetilde{C}\\\widetilde{C}^*&\widetilde{B}\end{array}\right),\end{equation*}
where $A$ is invertible. Then
\begin{equation*}\det M=\det A \det \widetilde{B}^{-1}.\end{equation*}
\end{lemma}

\medskip

\begin{proof}Since $A$ is invertible, we have
\begin{equation*}\left(\begin{array}{cc}I&O\\-C^*A^{-1}&I\end{array}\right)\left(\begin{array}{cc}A&C\\C^*&B\end{array}\right)\left(\begin{array}{cc}I&-A^{-1}C\\O&I\end{array}\right)=\left(\begin{array}{cc}A&O\\O&B-C^*A^{-1}C\end{array}\right),\end{equation*}
this implies
\begin{equation}\label{bmdet} \det M=\det A \det(B-C^*A^{-1}C).\end{equation}
$M$ is invertible, so $B-C^*A^{-1}C$ is also invertible.
At the same time, we have
\begin{equation*}M^{-1}=\left(\begin{array}{cc}I&-A^{-1}C\\O&I\end{array}\right)\left(\begin{array}{cc}A&O\\O&B-C^*A^{-1}C\end{array}\right)^{-1}\left(\begin{array}{cc}I&O\\-C^*A^{-1}&I\end{array}\right),\end{equation*}
which implies
\begin{equation}\label{bminv} \widetilde{B}=(B-C^*A^{-1}C)^{-1}.\end{equation}
The required equality is a combination of (\ref{bmdet}) and (\ref{bminv}).
\end{proof}

\medskip

{\bf A proof to Theorem \ref{mth}. } Let $\pi_Y$ and $\pi_{\mathbb{C}^m}$ be the two projections:
\begin{equation}\pi_Y:\mathbb{C}^m\times Y\rightarrow Y, \qquad \pi_Y(z,y)=y,\end{equation}
\begin{equation}\pi_{\mathbb{C}^m}:\mathbb{C}^m\times Y\rightarrow \mathbb{C}^m, \qquad \pi_{\mathbb{C}^m}(z,y)=z.\end{equation}
By K\"unneth's formula (see e.g. Section 5 of \cite{GTM82}) and the result on the de Rahm cohomology groups of $\mathbb{C}^m$
\begin{equation}H_{dR}^k(\mathbb{C}^m)=\left\{\begin{array}{ll}
\mathbb{R}, & k=0,\\0,, &k\geq 1,
\end{array}
\right.
\end{equation}
there exists a closed real 2-form $\Theta$ on $Y$, such that
\begin{equation}[\omega]=[\pi_Y^*\Theta],\end{equation}
in the sense of de Rahm cohomology classes.
For any $z\in \mathbb{C}^n$, denote the embedding from $Y$ to $\mathbb{C}^m\times Y$
\begin{equation}y\mapsto (z,y),\end{equation}
by $i_z$.
For distinct $z_1,z_2\in \mathbb{C}^m$, since $\pi_Y\circ i_{z_1}=id_Y=\pi_Y\circ i_{z_2}$, we have
\begin{equation}\label{eq_kc}[i_{z_1}^*\omega]=[\Theta]=[i_{z_2}^*\omega].\end{equation}
Obviously all $i_z^*\omega$ are K\"ahler forms on $Y$, so (\ref{eq_kc}) shows that all $i_z^*\omega$ are in the same K\"ahler class. By Calabi-Yau theorem (see \cite{Y78}), there is a unique K\"ahler form $\omega_Y$ in this K\"ahler class satisfying $Ric(\omega_Y)=Ric(\omega_{Y_0})$ and consequently
\begin{equation}\omega_Y^n=c\omega_{Y_0}^n\end{equation}
for some positive constant $c$.

Now we can write the conditions on $\omega$ as follow
\begin{itemize}
\item[1).]$C^{-1}(\omega_{\mathbb{C}^m}+\omega_{Y})\leq \omega \leq C(\omega_{\mathbb{C}^m}+\omega_{Y})$, for some positive constant $C$;
\item[2).]$\omega^{n+m}=c^{-1}(\omega_{\mathbb{C}^m}+\omega_{Y})^{m+n}$;
\item[3).]for any $z\in \mathbb{C}^m$, $i_{z}^*\omega$ and $\omega_Y$ are in the same K\"ahler class.
\end{itemize}

\medskip

Denote $g_0$ and $g$ to be the Riemann metric associated with $\omega_{\mathbb{C}^m}+\omega_Y$ and $\omega$ respectively and let $g^{-1}$ be the metric on $T^*(\mathbb{C}^m\times Y)$ induced by $g$. Let $\{z^{i}\}_{i=1}^{m}$ be the standard complex coordinate system on $\mathbb{C}^m$, for $i,j=1,2,\cdots,m$, define
\begin{equation}u^{i\bj}=\frac{1}{2}g^{-1}(dz^i,d\bar{z}^j),\qquad u_{i\bj}=2g\left(\frac{\partial}{\partial z^i},\frac{\partial}{\partial \bar{z}^j}\right).\end{equation}
For any point $(z,y)$, we choose a complex normal coordinates system $\{z^{\alpha}\}_{\alpha=m+1}^{m+n}$ around $y$ with respect to $\omega_Y$. Computing under the coordinate system $\{z^a\}_{a=1}^{m+n}$, we have
\begin{equation}\label{luij}\Delta_{\omega}u^{i\bj}=g^{a_1\overline{b_1}}g^{a_2\overline{b_2}}g^{i\overline{b_3}}g^{a_3\bar{i}}g_{a_1\overline{b_3} a_2}g_{\overline{b_1}a_3\overline{b_2}}+g^{i\overline{b_3}}g^{a_3\bar{j}}R_{a_3\overline{b_3}},\end{equation}
\begin{equation}\label{luij2}\Delta_{\omega} u_{i\bj}=g^{a_1\overline{b_1}}g^{a_2\overline{b_2}}g_{a_1\overline{b_2}i}g_{a_2\overline{b_1}\bj}
.\end{equation}
where $g^{a\bar{b}}, g_{a\bar{b}}, g_{a\bar{b}c}, R_{a\bar{b}}$ are the coefficients of components of $g^{-1},g,\nabla_{g_0}g,Ric(g_0)$ respectively, $a,b,a_k,b_k=1,2,\cdots,m+n$ ($k=1,2,3$).

It is clear that for $i,j=1,2,\cdots,m$, $g^{i\bj}=u^{i\bj}$ and $g_{i\bj}=u_{i\bj}$. Namely
\begin{equation*}(g_{a\bar{b}})=\left(\begin{array}{cc} u_{i\bj} & g_{\alpha\bj}\\g_{i\bar{\beta}} & g_{\alpha\bar{\beta}}\end{array}\right),\qquad(g^{a\bar{b}})=\left(\begin{array}{cc} u^{i\bj} & g^{\alpha\bj}\\g^{i\bar{\beta}} & g^{\alpha\bar{\beta}}\end{array}\right).\end{equation*}
Furthermore,
$(u^{i\bj})$ and $(u_{i\bj})$ are bounded and uniformly positive-defined.

Clearly $Ric(g_0)\geq0$, so (\ref{luij}) implies that $(\Delta_{\omega}u^{i\bj})\geq 0$, i.e. $(u^{i\bj})$ is subharmonic with respect to $\omega$. Applying Theorem \ref{mvhm} and Corollary \ref{mvhm1}, we can find some constant Hermitian matrix $A$ such that
\begin{equation}\label{lim_mtr} (u^{i\bj})\leq A,\end{equation}
everywhere, and
\begin{equation}\label{lim_det}\lim_{r\rightarrow \infty}\intav_{B^{\omega}_{r}(x_0,z_0)}\det(u^{i\bj})dV_g=\det A.\end{equation}
We can find some sufficiently large constant $C$ such that for $r\gg \diam Y$,
\begin{equation*}B_r(z_0)\times Y\subset B^{\omega}_{Cr}(z_0,y_0)\subset  B_{C^2r}(z_0)\times Y,\end{equation*}
so we have
\begin{equation*}0\leq \int_{B_r(z_0)\times Y}(\det A -\det(u^{i\bj}))dV_{g_0}\leq c^{-1}\int_{B^{\omega}_{Cr}(x_0,z_0)}(\det A -\det(u^{i\bj}))dV_g,\end{equation*}
and
\begin{equation*}\Vol_{g_0}(B_r(z_0)\times Y)=C^{-4n}\Vol(B_{C^2r}(z_0)\times Y)\geq C^{-4n}c^{-1}\Vol_{\omega}(B^{\omega}_{Cr}(z_0,y_0)).\end{equation*}

Then we can obtain
\begin{equation*}0\leq\intav_{B_r(z_0)\times Y}(\det A -\det(u^{i\bj}))dV_{g_{0}}\leq C^{-4n}\left(\intav_{B^{\omega}_{Cr}(z_0,y_0)}(\det A -\det(u^{i\bj}))dV_g\right),\end{equation*}
consequently
\begin{equation}\lim_{r\rightarrow \infty}\intav_{B_r(z_0)\times Y}\det(u^{i\bj}) dV_{g_{0}} =\det A.\end{equation}
By computing under the local coordinate system $\{z^a\}_{a=1}^{m+n}$ mentioned above and applying Lemma \ref{bmdetinv}, we can check that
\begin{equation}\label{vol}(\det(u^{i\bj}))^{-1}\frac{(i_z^*\omega)^n}{\omega_Y^n}=\frac{\omega^{n+m}}{(\omega_{Y}+\omega_{ \mathbb{C}^m})^{m+n}}=c^{-1},\end{equation}
at any point $(z,y)$. Then we have
\begin{equation}\intav_{\{z\}\times Y} \det(u^{i\bj})\omega_Y^n =\intav_{\{z\}\times Y} c (i_z^*\omega)^n=c.\end{equation}
This tells
\begin{equation}\label{Xav}\intav_{B_r(z_0)\times Y}\det(u^{i\bj})dV_{g_0}=c,\end{equation}
for any $z_0$ and $r>0$.
Together with (\ref{lim_mtr}) and (\ref{lim_det}), we have
\begin{equation}\det(u^{i\bj})\leq \det A =c ,\end{equation}
and then
\begin{equation}\det(u^{i\bj})=\det A=c.\end{equation}
Since $(u^{i\bj})$ is positive-defined, using (\ref{lim_mtr}) again, we obtain
\begin{equation}\label{con1}(u^{i\bj})\equiv A.\end{equation}
By ($\ref{vol}$), we have
\begin{equation}(i_z^*\omega)^n=\omega_Y^n\end{equation}for any $z\in\mathbb{C}^m$.
We already know that $i_z^*\omega$ and $\omega_Y$ are in the same K\"ahler class, so
\begin{equation}\label{con2}i_z^*\omega=\omega_Y.\end{equation}
For any $i=1,2,\cdots,m$, by (\ref{luij}) and (\ref{con1}) we have
\begin{equation}\label{cenn1}0=\Delta_\omega u^{i\bi}\geq C^{-1}\sum\limits_{a_1,a_2=1}^{m+n}|g_{a_1\bi a_2}|^2.\end{equation}
(\ref{cenn1}) implies that $u_{i\bi}$ is a constant function, then consequently
\begin{equation}
\label{cenn2}0=\Delta_{\omega} u_{i\bi}\geq g^{a_1\overline{b_1}}g^{a_2\overline{b_2}}g_{a_1\overline{b_2}i}g_{a_2\overline{b_1}\bi}\geq C^{-1}\sum\limits_{a,b=1}^{m+n}|g_{a\bar{b}i}|^2.
\end{equation}
(\ref{cenn2}) implies $(u_{i\bj})$ is a constant matrix.
At the same time, (\ref{con2}) implies
\begin{equation}\label{cenn3}\sum\limits_{a=1}^{m+n}\sum\limits_{\alpha,\beta=m+1}^{m+n}|g_{\alpha\bar{\beta}a}|^2=0.\end{equation}
(\ref{cenn1}), (\ref{cenn2}) and (\ref{cenn3}) together show $\nabla_{g_0}g=0$.

\medskip

We define
\begin{equation}\eta^i=i_z^*(\textstyle{\frac{\partial}{\partial z^i}}\lrcorner \omega),\end{equation}
where $z\in\mathbb{C}^m$, $i=1,2\cdots,m$. Since $\nabla_{g_0}g=0$, this definition doesn't depend on the choice of $z$ and every $\eta^i$ is an $\omega_Y$-paralleled (0,1)-form.
The expression (\ref{exs}) can be easily checked under a local coordinate system. This concludes the proof of Theorem \ref{mth}.

\hfill $\Box$ \\

\hspace{0.5cm}

\end{document}